\documentclass[12pt]{amsart}
\usepackage{amsmath}
\usepackage{amssymb}
\usepackage{comment}
\newtheorem{theorem}{Theorem}[section]
\newtheorem{corollary}[theorem]{Corollary}
\newtheorem{lemma}[theorem]{Lemma}

\theoremstyle{definition}

\theoremstyle{remark}
\newtheorem{remark}[theorem]{\bf{Remark}}
\theoremstyle{remark}
\newtheorem{example}[theorem]{\bf{Example}}
\numberwithin{equation}{section}

\newcommand{\beas} {\begin{eqnarray*}}
\newcommand{\eeas} {\end{eqnarray*}}
\newcommand{\bes} {\begin{equation*}}
\newcommand{\ees} {\end{equation*}}
\newcommand{\be} {\begin{equation}}
\newcommand{\ee} {\end{equation}}
\newcommand{\bea} {\begin{eqnarray}}
\newcommand{\eea} {\end{eqnarray}}

\newcommand{\R}{\mathbb R}

\newcommand{\C}{\mathbb C}

\newcommand{\what}{\widehat}

\begin{document}

\title[Converse of Fatou's theorem] {On Pointwise converse of Fatou's theorem for Euclidean and Real hyperbolic spaces}

\author{Jayanta Sarkar}

\address{(Jayanta Sarkar) Stat-Math Unit, Indian Statistical Institute, 203, B.T. Road, Kolkata-700108, India}
\email{jayantasarkarmath@gmail.com}
\thanks{The author is supported by a research fellowship from Indian Statistical Institute.}


\begin{abstract}
In this article, we extend a result of L. Loomis and W. Rudin, regarding boundary behavior of positive harmonic functions on the upper half space $\R_+^{n+1}$. We show that similar results remain valid for more general approximate identities. We apply this result to prove a result regarding boundary behavior of nonnegative eigenfunctions of the Laplace-Beltrami operator on real hyperbolic space $\mathbb H^n$.
We shall also prove a generalization of a result regarding large time behavior of solution of the heat equation proved in \cite{Re}. We use this result to prove a result regarding asymptotic behavior of certain eigenfunctions of the Laplace-Beltrami operator on real hyperbolic space $\mathbb H^n$.
\end{abstract}



\maketitle
\section{Introduction}
Given a complex measure (or a signed measure) $\mu$ on $\R^n$ its Poisson integral $P\mu$ on the upper half space $\R^{n+1}_+=\{(x,y):x\in\R^n,y>0\}$ is defined by the convolution $P\mu(x,y)=\mu\ast P_y(x)$, $x\in\R^n$, $y\in (0,\infty)$ where the Poisson kernel $P_y$ is given by
\begin{equation*}
P_y(x)=c_n\frac{y}{(y^2+\|x\|^2)^{\frac{n+1}{2}}},\:\:\:\:\: c_n=\pi^{-(n+1)/2}\Gamma\left(\frac{n+1}{2}\right).
\end{equation*}
All the measures that we shall consider in this article are assumed to be Borel measures.
Our motivation is a classical result of Fatou which relates the differentiability property of $\mu$ at a boundary point $x_0\in\R^n$ with the boundary behavior of $P\mu$ at $x_0$. Given a complex measure or a Radon measure $\mu$ on $\R^n$, the symmetric derivative of $\mu$ at a point $x_0\in\R^n$ is given by the limit
\begin{equation*}
D_{sym}\mu(x_0)=\lim_{r\to 0}\frac{\mu(B(x_0,r))}{m(B(x_0,r))},
\end{equation*}
provided, of course, that this limit exists, where $B(x_0,r)$ denotes ball of radius $r$ with center at $x_0$ with respect to the Euclidean metric and $m$ denotes the Lebesgue measure of $\R^n$. In this paper, by a Radon measure, we shall always mean a signed measure $\mu$ on $\R^n$ whose total variation $|\mu|$ is locally finite, that is, $|\mu|(K)$ is finite for every compact set $K\subset\R^n$.
For a complex measure or a Radon measure $\mu$ on $\R^n$ with well defined Poisson integral the result of Fatou says that if for some $x_0\in\R^n$ and $L\in \C$ we have $D_{sym}\mu(x_0)=L$ then
\begin{equation*}
\lim_{y\to 0}P\mu(x_0,y)=L.
\end{equation*}
A generalization of this result was proved in \cite{Sa} for more general approximate identities instead of the Poisson kernel (see Theorem \ref{fatour}).
A natural question in this regard is to ask about the validity of the converse implication : does
\begin{equation*}
\lim_{y\to 0}P\mu(x_0,y)=L,
\end{equation*}
imply that $D_{sym}\mu(x_0)=L$? It is important to note that the question is about the existence of limits at a single point $x_0$ and is not related to almost everywhere existence of the above limits. For $n=1$, it was shown by Loomis that this implication is false in general, but it remains true if $\mu$ is assumed to be a positive measure \cite{L}. Here and hereafter, a positive measure shall mean a nonnegative measure. Rudin \cite{Ru} generalized the result of Loomis for positive measures on $\R^n$ and proved the following.
\begin{theorem}\label{rud}
Suppose $\mu$ is a positive measure on $\R^n$ with well defined Poisson integral. If there exists $x_0\in\R^n$ and $L\in\R$ such that
\begin{equation*}
\lim_{y\to 0}P\mu(x_0,y)=L,
\end{equation*}
then $D_{sym}\mu(x_0)=L$.
\end{theorem}
Rudin's argument is based on an interesting application of Wiener's Tauberian theorem on the multiplicative group $(0,\infty)$. We refer the reader to \cite{BC, CD, Du, G1, Kh, Lo, UR} for related results. We now briefly describe the problems we are going to deal with in this paper. Suppose $\phi\in L^1(\R^n)$ be nonnegative with $\|\phi\|_1=1$. For $t>0$, we consider the usual approximate identity
\begin{equation*}
\phi_t(x)=t^{-n}\phi\left(\frac{x}{t}\right),\:\:\:\: x\in\R^n.
\end{equation*}
Suppose $\mu$ is a positive Radon measure on $\R^n$ such that the convolution $\mu\ast\phi_t$ is well defined for all small $t$. One can then ask under what conditions on the function $\phi$, the relation
\begin{equation*}
\lim_{t\to 0}\mu\ast \phi_t(x_0)=L,
\end{equation*}
implies that $D_{sym}\mu(x_0)=L$?
In this paper we will suggest a set of sufficient conditions on $\phi$ under which the above implication holds (see Theorem \ref{mt1}). Interestingly, it turns out that one of these conditions is also necessary (see Example \ref{nec1}). We will then use Theorem \ref{mt1} to prove a result analogous to that of Theorem \ref{rud} for nonnegative eigenfunctions of the Laplace-Beltrami operator on real hyperbolic space (see Theorem \ref{mtrealhyper}).

Another result which we will be discussed in this article was proved by Repnikov and Eidelman \cite{Re, Rp} regarding large time behavior of certain solutions of the heat equation. They proved, among other things, the following result.
\begin{theorem}\label{heatrep}
 Let $f\in L^{\infty}(\R^n)$, $x_0\in\R^n$, $L\in\C$. Then
 \begin{equation*}
 \lim_{r\to\infty}\frac{1}{m(B(x_0,r))}\int_{B(x_0,r)}f(x)\:dx=L,
 \end{equation*} if and only if
 \begin{equation*}
 \lim_{t\to\infty}f\ast h_t(x_0)=L,
 \end{equation*}
 where
 \begin{equation}\label{heat}
h_t(x)=(4\pi t)^{-\frac{n}{2}}e^{-\frac{\|x\|^2}{4t}},\:\:\:\: x\in\R^n,\:t>0,
\end{equation}
is the heat kernel.
\end{theorem}
We will extend the above theorem for two different approximate identities $\{\phi_t\}$ and $\{\psi_t\}$ (see Theorem \ref{mtheat}). Precisely, we will find a set of sufficient conditions on the function $\phi$ such that for $f\in L^{\infty}(\R^n)$, $x_0\in\R^n$ and $L\in\C$,
\begin{equation*}
\lim_{t\to\infty} f\ast \phi_t(x_0)=L
\end{equation*}
implies that
\begin{equation*}
\lim_{t\to\infty} f\ast \psi_{t}(x_0)=L.
\end{equation*}
We will further show that one of these conditions is also necessary. We will then use this result to prove a theorem regarding asymptotic behavior of certain eigenfunctions of the Laplace-Beltrami operator on real hyperbolic space (see Theorem \ref{largetime}).
As such there does not exist any connection between Theorem \ref{rud} and Theorem \ref{heatrep} apart from the fact that they seem complementary to each other in some sense. However, from our viewpoint, the main reason for including both these results here is the fact that proof of both these results depend crucially on the Wiener Tauberian theorem of the multiplicative group $(0,\infty)$.

The organization of the paper is as follows. In the next section we will prove Theorem \ref{mt1} and Theorem \ref{mtheat}. In the last section we will describe the necessary prerequisites for real hyperbolic spaces and prove Theorem \ref{mtrealhyper} and Theorem \ref{largetime}.

\section{Euclidean spaces}
In this paper, unless we explicitly mention, $\phi:\R^n\to (0,\infty)$ will always stand for an integrable, radial and radially decreasing function, that is,
\begin{equation*}
\phi (x)\geq \phi(y),\:\:\:\:\text{if}\:\:\|x\|\leq\|y\|,
\end{equation*}
with
\begin{equation*}
\int_{\R^n}\phi (x)dm(x)=1.
\end{equation*}
From now onwards, whenever an integral is involved, we will write $dx$ instead of $dm(x)$ and hope that it will not create any confusion.
In addition to the above, in some of our results we will also assume the following comparison condition on the function $\phi$.
\begin{equation}\label{comp}
\sup\:\left\{\frac{\phi_t(x)}{\phi (x)}\mid t\in (0,1), \|x\|>1\right\}<\infty.
\end{equation}
\begin{remark}\label{existence} It was proved in \cite[P. 137]{Sa} that if $\mu$ is a signed measure or a complex measure and $\phi$ is a nonnegative, radially decreasing function on $\R^n$ then finiteness of $|\mu|\ast \phi_{t_0}(x_0)$ implies the finiteness of $|\mu|\ast\phi_t(x)$ for all $t\in (0,t_0)$ and for all $x\in \R^n$. Note also that if $|\mu|(\R^n)$ is finite then $\mu\ast \phi_{t}(x)$ is well defined for all $(x,t)\in\R^{n+1}_+$.
\end{remark}
The following are some simple examples of functions which satisfy the condition (\ref{comp}).
\begin{example}\label{example}
\begin{enumerate}
\item [i)] For $\alpha\geq n/2$ and $\beta\geq 0$, we define
\begin{equation*}
K(x)=\frac{1}{(1+\|x\|^2)^\alpha\log(2+\|x\|^{\beta})},\:x\in\R^n.
\end{equation*}
We have for $t\in(0,1)$ and $\|x\|>1$,
\begin{eqnarray*}
	\frac{K_{t}(x)}{K(x)}&=& t^{-n}\frac{t^{2\alpha}(1+\|x\|^2)^\alpha\log(2+\|x\|^{\beta})}{(t^2+\|x\|^2)^\alpha\log\left(2+\frac{\|x\|^{\beta}}{t^{\beta}}\right)}\\
&\leq& t^{2\alpha-n}\frac{(1+\|x\|^2)^\alpha\log(2+\|x\|^{\beta})}{\|x\|^{2\alpha}\log\left(2+\|x\|^{\beta}\right)}\\
&\leq& t^{2\alpha-n}\left(1+\frac{1}{\|x\|^2}\right)^\alpha.
\end{eqnarray*}
This shows that $K$ satisfies the comparison condition (\ref{comp}). In fact, in case of $\alpha>n/2$, we have the following stronger result.
\begin{equation*}
\lim_{t\to 0}\frac{K_{t}(x)}{K(x)}=0,
\end{equation*}
uniformly for $x\in\R^n\setminus B(0,\epsilon)$ for any $\epsilon>0$. In particular,
\begin{equation*}
P(x)=\frac{\Gamma((n+1)/2)}{\pi^{(n+1)/2}}(1+\|x\|^2)^{-\frac{n+1}{2}},\:\:\:\:x\in\R^n
\end{equation*}
satisfies the comparison condition (\ref{comp}).
\item[ii)] For positive real numbers $\alpha$ and $\beta$ we define
\begin{equation*}
G(x)=e^{-\alpha\|x\|^{\beta}},\:x\in\R^n.
\end{equation*}
For for $t\in(0,1)$ and $\|x\|>1$,
\begin{equation*}
\frac{G_{t}(x)}{G(x)}=t^{-n}e^{-\alpha\left(\frac{1}{t^{\beta}}-1\right)\|x\|^{\beta}}
\leq t^{-n}e^{-\alpha\left(\frac{1}{t^{\beta}}-1\right)}
=e^{\alpha}t^{-n}e^{-\frac{\alpha}{t^{\beta}}}.
\end{equation*}
Taking limit as $t\to 0$, we get
\begin{equation*}
\lim_{t\to 0}\frac{G_{t}(x)}{G(x)}=0,
\end{equation*}
uniformly for $x\in\R^n\setminus\overline{B(0,1)}$. Thus, $G$ satisfies the comparison condition (\ref{comp}). In particular, the Gaussian 
\begin{equation*}
w(x)=(4\pi)^{-\frac{n}{2}}e^{-\frac{\|x\|^2}{4}},\:x\in\R^n
\end{equation*}
satisfies the comparison condition (\ref{comp}).
\end{enumerate}\end{example}
The following generalization of the result of Fatou was proved in \cite{Sa}.
\begin{theorem}\label{fatour}
Suppose $\phi:\R^n\to (0,\infty)$ satisfies the following conditions,
\begin{enumerate}
\item $\phi$ is radial, radially decreasing measurable function with $\|\phi\|_1=1$.
\item $\phi$ satisfies the condition (\ref{comp}).
\end{enumerate}
Suppose $\mu$ is a Radon measure (or a complex measure) on $\R^n$ such that $|\mu|\ast \phi_{t_0}(x_1)$ is finite for some $t_0>0$ and $x_1\in\R^n$. If for some $x_0\in \R^n$ and $L\in [0,\infty)$ we have $D_{sym}\mu(x_0)=L$ then
\begin{equation*}
\lim_{t\to 0} \mu\ast\phi_t(x_0)=L.
\end{equation*}
\end{theorem}
\begin{remark}
It was shown in \cite{Sa} that the theorem above fails in the absence of condition (\ref{comp}).
\end{remark}
Our main interest in this paper is to prove the converse implication for positive measures under appropriate hypothesis on the function $\phi$. Our first lemma shows that condition (\ref{comp}) can be used to reduce matters to the case of a finite positive measure $\mu$.
\begin{lemma}\label{reduce}
Suppose $\phi$ is as in Theorem \ref{fatour}. If $\mu$ is a positive Radon measure such that $\mu\ast\phi_{t_0}(0)$ is finite for some $t_0\in (0,\infty)$, then
\begin{equation}\label{reducemu}
\lim_{t\to 0}\mu\ast\phi_t(0)=\lim_{t\to 0}\tilde{\mu}\ast\phi_t(0),
\end{equation}
where $\tilde{\mu}$ is the restriction of $\mu$ on the closed ball $\overline{B(0,t_0)}$. Moreover,
\begin{equation}\label{reduced}
D_{sym}\mu(0)=D_{sym}\tilde{\mu}(0).
\end{equation}
\end{lemma}
\begin{proof}
We write for $t\in (0,t_0)$,
\begin{eqnarray}
\mu\ast\phi_t(0)&=& \int_{\{x\in\R^n\mid \|x\|\leq t_0 \}}\phi_t(x)\:d\mu(x)+ \int_{\{x\in\R^n\mid \|x\|>t_0 \}}\phi_t(x)\:d\mu(x)\nonumber\\
&=&\tilde{\mu}\ast\phi_t(0)+\int_{\{x\in\R^n\mid \|x\|>t_0 \}}\phi_t(x)\:d\mu (x).\label{rel1}
\end{eqnarray}
Since $\phi$ is a radial function, we will write for the sake of simplicity $\phi(x)=\phi(r)$, whenever $\|x\|=r$. For any $r\in (0,\infty)$, we have
\begin{equation*}
\int_{r/2\leq \|x\|\leq r}\phi (x)\:dx\geq \omega_{n-1}\phi (r)\int_{r/2}^rs^{n-1}\:ds=A_nr^n\phi (r),
\end{equation*}
where $\omega_{n-1}$ is the surface area of the unit sphere $S^{n-1}$ and $A_n$ is a positive constant which depends only on the dimension.
Since $\phi$ is an integrable function, the integral on the left hand side converges to zero as $r$ goes to zero and infinity. Hence, it follows that
\begin{equation}
\lim_{\|x\|\to 0}\|x\|^n\phi (x)=\lim_{\|x\|\to \infty}\|x\|^n\phi (x)=0\label{rel2}.
\end{equation}
We denote the integral appearing on the right-hand side of (\ref{rel1}) by $I(t)$. Then, for $t\in (0,1)$
\begin{eqnarray}
I(tt_0)&=&(tt_0)^{-n}\int_{\{x\in\R^n\mid \|x\|>t_0 \}}\phi\left( \frac{x}{tt_0} \right)\:d\mu (x)\nonumber\\
&=&\int_{\{x\in\R^n\mid \|x\|>t_0 \}}\frac{\left(\frac{\|x\|}{tt_0}\right)^n\phi\left( \frac{x}{tt_0} \right)}{\|x\|^n\phi_{t_0}(x)}\phi_{t_0} (x)\:d\mu(x)\label{rel3}.
\end{eqnarray}
From (\ref{rel2}) we get that
\begin{equation*}
\lim_{t\to 0}\left(\frac{\|x\|}{tt_0}\right)^n\phi\left( \frac{x}{tt_0} \right)=0,
\end{equation*}
for each fixed $x\in\R^n$ and by (\ref{comp})
\begin{equation*}
\frac{\left(\frac{\|x\|}{tt_0}\right)^n\phi\left( \frac{x}{tt_0} \right)}{\|x\|^n\phi_{t_0}(x)}=\frac{\phi_t\left(\frac{x}{t_{0}}\right)}{\phi\left(\frac{x}{t_0}\right)}\leq C,\:\:\:\:\|x\|>t_0,\: 0<t<1,
\end{equation*}
for some positive constant $C$. Since $\phi_{t_0}\in L^1(\R^n,d\mu)$, it follows from (\ref{rel3}), by the dominated convergence theorem that
\begin{equation*}
\lim_{t\to 0}I(tt_0)=0.
\end{equation*}
Consequently,
\begin{equation*}
\lim_{t\to 0}I(t)=\lim_{t\to 0} I(tt_0^{-1}t_0)=\lim_{t\to 0}\int_{\{x\in\R^n\mid \|x\|>t_0 \}}\phi_t(x)\:d\mu (x)=0.
\end{equation*}
This proves (\ref{reducemu}). From the definition of $\tilde{\mu}$, we have
\begin{equation*}
\tilde{\mu}(B(0,r))=\mu(B(0,r)),
\end{equation*}
for all $r\in(0,t_0)$. This proves (\ref{reduced}).
\end{proof}
Next, we are going to prove two simple lemmas which will be used in the proof of the main theorem.
\begin{lemma}\label{boundedM}
Let $\phi$ be a strictly positive, radial and radially decreasing function on $\R^n$. Let $\mu$ be a finite positive measure on $\R^n$ and \begin{equation*}
v(t)=\mu\ast\phi_t(0),\:\:\:\: t\in (0,\infty).
\end{equation*}If
\begin{equation*}
\lim_{t\to 0}v(t)=L<\infty,
\end{equation*}
then
\begin{enumerate}
\item
$v$ is a bounded function on $(0,\infty)$.
\item The function
\begin{equation*}
M(r)=\frac{\mu (B(0,r))}{m(B(0,r))},\:\:\:\: r\in (0,\infty),
\end{equation*}
is a bounded function on $(0,\infty)$.
\end{enumerate}
\end{lemma}
\begin{proof}
The proof of $(1)$ is simple. Since $v$ has a finite limit $L$ at zero there exists a $\delta >0$ such that
\begin{equation*}
0\leq v(t)\leq L+1,
\end{equation*}
for all $0<t\leq\delta$. On the other hand, for all $t\geq\delta$
\begin{equation*}
v(t)=\int_{\R^n}\phi_t(x)\:d\mu(x)
= t^{-n}\int_{\R^n}\phi\left(\frac{x}{t}\right)d\mu(x)
\leq \delta^{-n}\phi(0)\mu(\R^n).
\end{equation*}
Hence, $v$ is bounded on $(0,\infty)$.

As $\mu$ is a finite positive measure it is clear that $M$ is bounded for all large $r$. Hence, to prove $(2)$ it suffices to show that $M$ is bounded for $r$ near zero. We observe that for any $r\in (0,\infty)$,
\begin{equation*}
\mu\ast\phi_r(0)\geq r^{-n}\int_{B(0,r)}\phi\left(\frac{x}{r}\right)\:d\mu(x)\geq r^{-n}\int_{B(0,r)}\phi(1)\:d\mu(x)=m(B(0,1))\phi(1)M(r).
\end{equation*}
Setting $C_{n,\phi}=\left(m(B(0,1))\phi(1)\right)^{-1}$, the equation above implies that
\begin{equation}\label{elminq5}
M(r)\leq C_{n,\phi}v(r),\:\:\:\: r\in (0,\infty).
\end{equation}
Since $v(r)$ is bounded for all $r$ near zero it follows from above that so is $M(r)$.
\end{proof}
\begin{remark}\label{elminq6}
We observe that the inequality (\ref{elminq5}) remains valid even if $\mu$ is an infinite positive measure. This observation will be used in the proof of Theorem \ref{mt2}.
\end{remark}
To prove our next lemma we will have to use the convolution on the multiplicative group $(0,\infty)$ with Haar measure $ds/s$. To differentiate with the convolution on $\R^n$ we will write
\begin{equation*}
f\ast_{(0,\infty)}g(t)=\int_0^{\infty}f(s)g\left(\frac{t}{s} \right)\frac{ds}{s},
\end{equation*}
where $f$ and $g$ are integrable on $(0,\infty )$ with respect to Haar measure $ds/s$.
\begin{lemma}\label{limitlemma}
Suppose $k\in L^{\infty}(0,\infty)$ is such that
\begin{equation*}
\lim_{t\to 0}k(t)=L,
\end{equation*}
for some $L\in\C$. Then for all $f\in L^1((0,\infty), dt/t)$ with 
\begin{equation*}
\int_0^{\infty}f(s)\:\frac{ds}{s}=1, 
\end{equation*}
we have
\begin{equation*}
\lim_{t\to 0}f\ast_{(0,\infty)}k(t)=L.
\end{equation*}
\end{lemma}
\begin{proof}
Let $f$ be as above. Note that for each $t\in(0,\infty)$,
\begin{eqnarray}\label{fconv}
\left|f\ast_{(0,\infty)}k(t)-k(t)\right|
&=& \left|\int_{0}^{\infty}f(s)k\left(\frac{t}{s}\right)\:\frac{ds}{s}-\int_{0}^{\infty}f(s)k(t)\:\frac{ds}{s}\right|\nonumber\\
&\leq & \int_{0}^{\infty}\left|f(s)\right|\left|k\left(\frac{t}{s}\right)-k(t)\right|\:\frac{ds}{s}.
\end{eqnarray}
Since $k$ has limit $L$ at zero, it follows that for each fixed $s\in (0,\infty)$,
\begin{equation*}
\lim_{t\to 0}\left|k\left(\frac{t}{s}\right)-k(t)\right|=0.
\end{equation*}
The integrand on the right-hand side of (\ref{fconv}) is bounded by $2\|k\|_{L^{\infty}(0,\infty)}|f|$, an integrable function on $(0,\infty)$.
Using dominated convergence theorem we conclude from (\ref{fconv}) that
\begin{equation*}
\lim_{t\to 0}\left|f\ast_{(0,\infty)}k(t)-k(t)\right|=0,
\end{equation*}
which in turn, implies that
\begin{equation*}
\lim_{t\to 0}f\ast_{(0,\infty)}k(t)=\lim_{t\to 0}k(t)=L.
\end{equation*}
\end{proof}
To proceed further, we will need the following versions of Wiener's Tauberian theorem \cite[Theorem 9.7]{Ruf} for the multiplicateive group $(0,\infty)$.
\begin{theorem}\label{wtt}
Suppose $\psi\in L^{\infty}(0,\infty)$ and $K\in L^1((0,\infty),dt/t)$ with the Fourier transform $\hat{K}$ everywhere nonvanishing on $\R$.
\begin{enumerate}
\item If, $\lim_{t\to\infty}K\ast_{(0,\infty)}\psi (t)=a\hat{K}(0)$, then for all $f\in L^1((0,\infty), dt/t)$, $\lim_{t\to\infty}f\ast_{(0,\infty)} \psi (t)=a\hat{f}(0)$.
\item If, $\lim_{t\to 0}K\ast_{(0,\infty)}\psi (t)=a\hat{K}(0)$, then for all $f\in L^1((0,\infty), dt/t)$, $\lim_{t\to 0}f\ast_{(0,\infty)} \psi (t)=a\hat{f}(0)$.
\end{enumerate}
\end{theorem}
We are now in a position to present a generalization of Theorem \ref{rud}, which is the main result of this paper. 
\begin{theorem}\label{mt1}
Suppose $\phi:\R^n\to (0,\infty)$ satisfies the following conditions,
\begin{enumerate}
\item $\phi$ is radial, radially decreasing measurable function with $\|\phi\|_1=1$.
\item $\phi$ satisfies the condition (\ref{comp}).
\item For all $y\in\R$,
\begin{equation}\label{tauberian}
\int_{\R^n}\phi(x)\|x\|^{iy}\:dx\neq 0.
\end{equation}
\end{enumerate}
Suppose $\mu$ is a positive Radon measure on $\R^n$ such that $\mu\ast\phi_{t_0}(0)$ is finite for some $t_0\in(0,\infty)$. If for some $x_0\in \R^n$ and $L\in [0,\infty)$
\begin{equation*}
\lim_{t\to 0} \mu\ast\phi_t(x_0)=L,
\end{equation*}
then $D_{sym}\mu(x_0)=L$.
\end{theorem}
\begin{proof}
Without loss of generality, we can assume $x_0=0$. Indeed, we consider the translated measure $\mu_0=\tau_{-x_0}\mu$, where
\begin{equation*}
\tau_{x_0}\mu (E)=\mu(E-x_0),
\end{equation*}
for all Borel subsets $E\subset \R^n$. Using translation invariance of the Lebesgue measure it follows from the definition of symmetric derivative that $D_{sym}\mu_0(0)$ and $D_{sym}\mu (x_0)$ are equal. As translation commutes with convolution, it also follows that
\begin{equation*}
\mu_0\ast\phi_t(0)=(\tau_{-x_0}\mu\ast\phi_t)(0)=\tau_{-x_0}(\mu\ast\phi_t )(0)=\mu\ast\phi_t(x_0),
\end{equation*}
for any $t\in (0,\infty)$. Applying Lemma \ref{reduce}, we can restrict $\mu$ on $\overline{B(0,t_0)}$, if necessary, to assume that $\mu$ is a finite positive measure. As before, we define
\begin{equation*}
v(t)=\mu\ast\phi_t(0),\:\:0<t<\infty,\:\:\:\:\:\: M(r)=\frac{\mu (B(0,r)}{m(B(0,r))},\:0<r<\infty.
\end{equation*}
By Lemma \ref{boundedM}, we know that both $v$ and $M$ are bounded functions on $(0,\infty)$.
Following \cite{Ru}, we consider the following function on the multiplicative group $(0,\infty)$.
\begin{equation*}
H(t)=
\begin{cases}
0,&t\in(0,1)\\nt^{-n},&t\geq 1.
\end{cases}
\end{equation*}
Clearly, $H\in L^1((0,\infty),dt/t)$ with $\|H\|_{L^1((0,\infty),dt/t)}=1$. We observe that for $r\in (0,\infty)$,
\begin{eqnarray}
H\ast_{(0,\infty)}v(r)&=&\int_{0}^{\infty}H\left(\frac{r}{s}\right)v(s)\:\frac{ds}{s}
\nonumber\\
&=&n\int_{0}^{r}\left(\frac{s}{r}\right)^n\mu\ast\phi_s(0)\:\frac{ds}{s}
\nonumber\\
&=& nr^{-n}\int_{0}^{r}s^n\int_{\R^n}s^{-n}\phi\left(\frac{x}{s}\right)\:d\mu(x)\:\frac{ds}{s}
\nonumber\\
&=& nr^{-n}\int_{\R^n}\int_{0}^{r}\phi\left(\frac{x}{s}\right)\:\frac{ds}{s}\:d\mu(x)\:\:\:\:\text{(as $\phi\geq 0$)}.
\label{hvonvmain}
\end{eqnarray}
Since $M$ is a bounded function, it follows that
\begin{equation*}
\lim_{r\to 0}\mu(B(0,r))=\lim_{r\to 0}m(B(0,r))M(r)=0,
\end{equation*}
that is, $\mu$ has no point mass at $0$. We can now write, from (\ref{hvonvmain}), the convolution $H\ast_{(0,\infty)}v$ in a different way.
\begin{eqnarray}\label{hconv1}
H\ast_{(0,\infty)}v(r)&=& nr^{-n}\int_{\R^n\backslash\{0\}}\int_{0}^{r}\phi\left(\frac{x}{s}\right)\:\frac{ds}{s}\:d\mu(x)\nonumber\\
&=&nr^{-n}\int_{\R^n\backslash\{0\}}\int_{0}^{\frac{r}{\|x\|}}\phi\left(\frac{x}{t\|x\|}\right)\:\frac{dt}{t}\:d\mu(x)\nonumber\\
&&\:\:\:\:\:\:\:\:\text{(change of variables, $s=t\|x\|$)}\nonumber\\
&=&nr^{-n}\int_{\R^n\backslash\{0\}}\int_{0}^{\frac{r}{\|x\|}}\phi\left(\frac{1}{t}\right)\:\frac{dt}{t}\:d\mu(x)\nonumber\\
&=&n\int_{0}^{\infty}\int_{B(0,\frac{r}{t})\backslash\{0\}}\:d\mu(x)\phi\left(\frac{1}{t}\right)t^{-n}\left(\frac{t}{r}\right)^n\:\frac{dt}{t}\nonumber\\
&&\:\:\:\:\:\:\text{(by Fubini's theorem)}\nonumber\\
&=&n\:m(B(0,1))\int_{0}^{\infty}\frac{\mu\left(B\left(0,\frac{r}{t}\right)\right)}{m\left(B\left(0,\frac{r}{t}\right)\right)}t^{-n}\phi\left(\frac{1}{t}\right)\:\frac{dt}{t}\nonumber\\
&=&n\:m(B(0,1))\int_{0}^{\infty}M\left(\frac{r}{t} \right)t^{-n}\phi\left(\frac{1}{t}\right)\:\frac{dt}{t}.
\end{eqnarray}
By defining
\begin{equation*}
g(s)=n\:m(B(0,1))s^{-n}\phi\left(\frac{1}{s}\right),\:0<s<\infty,
\end{equation*}
equation (\ref{hconv1}) can be rewritten as
\begin{equation*}
H\ast_{(0,\infty)}v(r)=M\ast_{(0,\infty)}g(r),\:\:\:\: r\in (0,\infty).
\end{equation*}
By hypothesis, $v(r)$ converges to $L$ as $r$ goes to zero and hence by Lemma \ref{limitlemma}, so does $H\ast_{(0,\infty)}v(r)$ . It now follows from the equation above that
\begin{equation}\label{taub}
\lim_{r\to 0} M\ast_{(0,\infty)}g(r)=L.
\end{equation}
We want to use Theorem \ref{wtt} to deduce from (\ref{taub}) that for all $f\in L^1\big((0,\infty),\frac{ds}{s}\big)$ with integral one,
\begin{equation}\label{weakconv}
\lim_{r\to 0} M\ast_{(0,\infty)}f(r)=L.
\end{equation}
In order to do so we need to show that $g$ is of integral one with everywhere nonvanishing Fourier transform on the multiplicative group $(0,\infty)$. This can be deduced from the assumption (\ref{tauberian}). Precisely, for all $y\in\R$ we have
\begin{eqnarray*}
\what g(y)
&=&\int_{0}^{\infty}g(s)s^{-iy}\:\frac{ds}{s}\\
&=&n\:m(B(0,1))\int_{0}^{\infty}s^{-n}\phi\left(\frac{1}{s}\right)s^{-iy}\:\frac{ds}{s}\\
&=& \omega_{n-1}\int_{0}^{\infty}\phi(t)t^{iy}t^n\:\frac{dt}{t}\\
&&\:\:\:\:\text{( change of variables, $t=1/s$)}\\
&=&\int_{\R^n}\phi(x)\|x\|^{iy}\:dx.
\end{eqnarray*}
We observe that by considering $y=0$, it also follows that $g\in L^1((0,\infty),dt/t)$ with $\hat{g}(0)=1$. Validity of the limit (\ref{weakconv}) now follows from Wiener's Tauberian theorem. In the final part of the proof we shall use (\ref{weakconv}) to deduce that
\begin{equation}\label{ultimate}
D_{sym}\mu(0)=\lim_{r\to 0}M(r)=L.
\end{equation}
We fix an arbitrary $\gamma\in (1,\infty)$. We choose positive functions $f_i\in C_c(0,\infty)$ with  $\|f_i\|_{L^1((0,\infty),dt/t)} =1$ for $i=1,2$, and
\begin{equation*}
\text{supp}~f_1\subset [1,\gamma],\:\:\:\:\:\text{supp}~f_2\subset \left[\frac{1}{\gamma},1\right].
\end{equation*}
By monotonicity of $\mu$ we have for $t\in [1,\gamma]$ and $r\in (0,\infty)$,
\begin{equation*}
m\left(B\left(0,\frac{r}{t}\right)\right)M\left(\frac{r}{t}\right)=\mu\left(B\left(0,\frac{r}{t}\right)\right)\leq\mu(B(0,r)),
\end{equation*}
and hence
\begin{equation}\label{Mapprox1}
M\left(\frac{r}{t}\right)\leq \frac{\mu(B(0,r))}{m\left(B\left(0,\frac{r}{t}\right)\right)}= t^nM(r)\leq\gamma^nM(r).
\end{equation}
By a similar argument it follows that for $t\in [1/\gamma,1]$ and $r\in (0,\infty)$,
\begin{equation}\label{Mapprox2}
M\left(\frac{r}{t}\right)\geq\gamma^{-n}M(r).
\end{equation}
Now, for $r\in (0,\infty)$
\begin{equation}\label{Mlimit1}
 M\ast_{(0,\infty)}f_1(r)=\int_{1}^{\gamma}f_1(t)M\left(\frac{r}{t}\right)\:\frac{dt}{t}\leq\int_{1}^{\gamma}\gamma^nM(r)f_1(t)\:\frac{dt}{t}=\gamma^nM(r),
\end{equation}
where the inequality follows from (\ref{Mapprox1}). Similarly, using (\ref{Mapprox2}) we get
\begin{equation}\label{Mlimit2}
M\ast_{(0,\infty)}f_2(r)\geq\gamma^{-n}M(r)\:\:\:\: r\in (0,\infty).
\end{equation}
Combining (\ref{Mlimit1}) and (\ref{Mlimit2}) we get
\begin{equation*}
\gamma^{-n}M\ast_{(0,\infty)}f_1(r)\leq M(r)\leq\gamma^{n}M\ast_{(0,\infty)}f_2(r),\:\:\:\:\:r\in (0,\infty).
\end{equation*}
Allowing $r$ tending to zero in the inequality above and using (\ref{weakconv}) we get
\begin{equation*}
\gamma^{-n}L\leq\liminf_{r\to 0}M(r)\leq\limsup_{r\to 0}M(r)\leq\gamma^nL.
\end{equation*}
This implies (\ref{ultimate}) as $\gamma>1$ is arbitrary. This completes the proof.
\end{proof}
We now show by an example that Theorem \ref{mt1} fails in the absence of condition (\ref{tauberian}).
\begin{example}\label{nec1}
Suppose $\phi:\R^n\to(0,\infty)$ is such that it satisfies the first two conditions of Theorem \ref{mt1} but does not satisfy the third condition. That is, there exists $y_0\in\R$, such that
\begin{equation}\label{fails}
\int_{\R^n}\phi (x)\cos (y_0\log \|x\|)\:dx=\int_{\R^n}\phi (x)\sin (y_0\log \|x\|)\:dx=0.
\end{equation}
As $\|\phi\|_1=1$ and $\phi$ is strictly positive, we have $y_0\neq 0$. We consider the function
\begin{eqnarray*}
f(x)&=& 2+\cos (y_0\log \|x\|), \:\:\:\: x\in\R^n\setminus \{0\},\\
&=& 1,\:\:\:\:\:\:\:\:\:\: x=0,
\end{eqnarray*}
and define a positive measure, $d\mu (x)=f(x)dx$. We will show that
\begin{equation}\label{limit}
\lim_{t\to 0} \mu\ast \phi_t(0)=2,
\end{equation}
but the symmetric derivative of $\mu$ does not exist at zero. Now, for all $t\in (0,\infty)$
\begin{eqnarray*}
\mu\ast\phi_t(0)&=& t^{-n}\int_{\R^n}\phi\left( \frac{x}{t}\right)f(x)\:dx\\
&=& \int_{\R^n}\phi (x)f(tx)\:dx\\
&=& 2\int_{\R^n}\phi (x)dx+\int_{\R^n}\phi (x)\cos (y_0\log t+y_0\log \|x\|)\:dx\\
&=& 2+\cos (y_0\log t)\int_{\R^n}\phi (x)\cos (y_0\log \|x\| )\:dx\\
&&\:\:\:\:\:-\sin (y_0\log t)\int_{\R^n}\phi (x)\sin (y_0\log \|x\| )\:dx\\
&=& 2,
\end{eqnarray*}
where the last equality follows from (\ref{fails}). This implies the limit (\ref{limit}). On the other hand, for $r\in (0,\infty)$
\begin{eqnarray*}
\frac{1}{m(B(0,r))}\int_{B(0,r)}f(x)\:dx&=&\frac{1}{m(B(0,r))}\int_{B(0,r)}(2+\cos(y_0\log\|x\|))\:dx\\
&=& 2+\frac{1}{m(B(0,r))}\text{Re}\left[\int_{B(0,r)}\|x\|^{iy_0}\:dx\right]\\
&=& 2+\frac{\omega_{n-1}}{m(B(0,1))}\text{Re}\left( \frac{r^{iy_0}}{n+iy_0} \right).
\end{eqnarray*}
Using the fact that $y_0$ is nonzero, it is easy to construct two different sequences $\{r_k\}$ converging to zero such that $\text{Re}( r_k^{iy_0}/(n+iy_0))$ converges to different limits. Hence, $D_{sym}\mu(0)$ does not exist.

It now remains to construct a function $\phi$ on $\R^n$, as above. To do this, we first consider the following functions defined on $(0,\infty)$,
\begin{equation*}
g(r)=\chi_{[e^{-1},e]}(r),\:\:\:\:\:\:\:
f(r)= \frac{r^n}{(1+r)^{2n}}.
\end{equation*}
Clearly, $f$ and $g$ both are in $L^1((0,\infty),dr/r)$ and the function $r\mapsto r^{-n}f(r)$ is decreasing in $(0,\infty)$. Moreover, for all $y\in\R\setminus \{0\}$
\begin{equation*}
\int_0^{\infty}g(r)r^{iy}\:\frac{dr}{r}=\frac{2\sin y}{y},
\end{equation*}
which vanishes for $y=\pi$. We now define
\begin{equation*}
\psi(s)=f\ast_{(0,\infty)} g(s)=\int_0^{\infty}f\left(\frac{s}{r}\right)g(r)\:\frac{dr}{r},\:\:\:\: s\in (0,\infty).
\end{equation*}
Then $\psi\in L^1((0,\infty), ds/s)$ with $\widehat{\psi}(\pi)=0$ and is strictly positive. Hence, 
\begin{equation*}
c_{\psi}:=\int_{0}^{\infty}\psi(s)\:\frac{ds}{s}>0.
\end{equation*} 
As the function $r\mapsto r^{-n}f(r)$ is decreasing in $(0,\infty)$ and $g$ is nonnegative, it follows that
\begin{equation}\label{dec}
\frac{\psi(s)}{s^n}=\int_0^{\infty}\frac{f(s/r)}{(s/r)^n}g(r)r^{-(n+1)}\:dr,
\end{equation}
is also decreasing. Moreover, for each $r\in (0,\infty)$
\begin{equation*}
\lim_{s\to 0}\frac{f(s/r)}{s^n}=\lim_{s\to 0}\frac{r^n}{(r+s)^{2n}}=r^{-n},
\end{equation*}
and
\begin{equation*}
\frac{f(s/r)}{s^n}g(r)\leq r^{-n}g(r),\:\:\:\: s\in (0,\infty).
\end{equation*}
From the expression (\ref{dec}) it follows by applying dominated convergence theorem that
\begin{equation*}
\lim_{s\to 0}\frac{\psi (s)}{s^n}=\int_0^{\infty}g(r)r^{-(n+1)}dr>0.
\end{equation*}
Finally, we define $\phi:\R^n\to (0,\infty)$ by
\[
\phi(x)=
\left\{
\begin{array}{ccl}
\displaystyle \frac{1}{c_{\psi}\omega_{n-1}}\frac{\psi (\|x\|)}{\|x\|^n}, && x\neq 0,\\
\displaystyle \frac{1}{c_{\psi}\omega_{n-1}}\lim_{s\to 0}\frac{\psi (s)}{s^n}, && x= 0.
\end{array}
\right.
\]
By construction, $\phi$ is strictly positive, radial and radially decreasing on $\R^n$. For all $y\in\R$,
\begin{equation*}
\int_{\R^n}\phi (x)\|x\|^{iy}\:dx=\omega_{n-1}\frac{1}{c_{\psi}\omega_{n-1}}\int_0^{\infty}\frac{\psi (s)}{s^n}s^{n-1+iy}\:ds=\frac{1}{c_{\psi}}\int_0^{\infty}\psi (s) s^{iy}\:\frac{ds}{s}.
\end{equation*}
This shows that $\phi\in L^1(\R^n)$,
\begin{equation*}
\int_{\R^n}\phi (x)\:dx=1,
\end{equation*}
and
\begin{equation*}
\int_{\R^n}\phi (x)\|x\|^{i\pi}dx=0.
\end{equation*}
It remains to show that $\phi$ satisfies the comparison condition (\ref{comp}). The function $\psi$ has the following explicit expression
\begin{equation}\label{psiexp}
\psi (s)=\int_{e^{-1}}^e\frac{r^ns^n}{(r+s)^{2n}}\:\frac{dr}{r}.
\end{equation}
Differentiation under the integral sign yields
\begin{equation*}
\psi'(s)=\int_{e^{-1}}^e\frac{n(r+s)^{2n-1}s^{n-1}(r-s)}{(r+s)^{4n}}r^{n-1}\:dr,
\end{equation*}
which is negative if $s\in (e,\infty)$. Hence, $\psi$ is decreasing in $(e,\infty)$. Now, for $\|x\|\in (e,\infty)$ and $t\in (0,1)$ we have from the definition of $\phi$
\begin{equation}\label{bdd1}
\frac{\phi_t(x)}{\phi (x)}= \frac{\psi(\|x\|/t)}{\psi (\|x\|)}\leq \frac{\psi(\|x\|)}{\psi (\|x\|)}=1.
\end{equation}
We will now deal with the case $\|x\|\in (1,e]$ and $t\in (0,1)$. Using the expression (\ref{psiexp}) of $\psi$ it follows that for $s\in (1,e]$, 
 \begin{equation}\label{bdd2}
\psi(s)\geq \int_{e^{-1}}^e\frac{e^{-(n-1)}}{(2e)^{2n}}\:dr=a_n>0.
\end{equation}
Also, for $t\in (0,1)$ and $s\in (1,e]$
\begin{equation}\label{bdd3}
\psi\left( \frac{s}{t} \right)=\int_{e^{-1}}^e\frac{t^ns^nr^{n-1}}{(tr+s)^{2n}}\:dr\leq \int_{e^{-1}}^es^{-n}r^{n-1}\:dr\leq \int_{e^{-1}}^er^{n-1}\:dr=b_n.
\end{equation}
It now follows from (\ref{bdd2}) and (\ref{bdd3}) that for $\|x\|\in (1,e]$, $t\in (0,1)$
\begin{equation}\label{bdd4}
\frac{\phi_t(x)}{\phi (x)}=\frac{\psi(\|x\|/t) }{\psi (\|x\|)}\leq a_n^{-1}b_n.
\end{equation}
Inequalities (\ref{bdd1}) and (\ref{bdd4}) together imply that $\phi$ satisfies the condition (\ref{comp}).
\end{example}
In the following theorem we show that Theorem \ref{mt1} remains valid for a restricted class of measures in the absence of condition $(2)$.
\begin{theorem}\label{mt2}
Suppose $\phi:\R^n\to (0,\infty)$ satisfies the following conditions,
\begin{enumerate}
	\item $\phi$ is radial, radially decreasing measurable function with $\|\phi\|_1=1$.
	\item $\phi$ satisfies the condition (\ref{tauberian}).
\end{enumerate}
Suppose $\mu$ is a positive Radon measure on $\R^n$ such that
\begin{equation}\label{growth}
\mu(B(0,r))=O(r^n),\:\: as\:\:r\to\infty,
\end{equation}
and that $\mu\ast\phi_{t_0}(0)$ is finite for some $t_0\in(0,\infty)$.
If for some $x_0\in \R^n$ and $L\in [0,\infty)$
\begin{equation*}
\lim_{t\to 0} \mu\ast\phi_t(x_0)=L,
\end{equation*}
then $D_{sym}\mu(x_0)=L$.	
\end{theorem}
\begin{proof}
Without loss of generality, we assume that $x_0=0$. We will use the same notation as in the proof of Theorem \ref{mt1}. From the proof of Theorem \ref{mt1} we observe that it suffices to prove the boundedness of the functions $v$ and $M$ and then the
the rest of the arguments remains same. 
Using Remark \ref{existence}, we note that $v$ is well defined in $(0,t_0]$ and the fact that $v$ is well defined in $(t_0,\infty)$ will be shown to be a consequence of the condition (\ref{growth}).
According to Remark \ref{elminq6}, the boundedness of $v$ implies boundedness of $M$. Therefore, it suffices to prove that under the hypothesis of the theorem the function $v$ is bounded.  Since $v$ has limit $L$ at zero, there exists $\delta\in(0,t_0)$, such that $v$ is bounded on $(0,\delta]$.
We observe that for $t>\delta$,
\begin{eqnarray}
\mu\ast\phi_t(0)&=&t^{-n}\int_{B(0,t)}\phi\left(\frac{x}{t}\right)\:d\mu(x)+t^{-n}\sum_{k=1}^{\infty}\int_{2^{k-1}t\leq \|x\|<2^kt}\phi\left(\frac{x}{t}\right)\:d\mu(x)\nonumber\\
&\leq&\phi(0)m(B(0,1))\frac{\mu(B(0,t))}{m(B(0,t))}+\sum_{k=1}^{\infty}\phi(2^{k-1})t^{-n}\mu(B(0,2^kt))\nonumber\\
&=&\phi(0)m(B(0,1))\frac{\mu(B(0,t))}{m(B(0,t))}+m(B(0,1))\sum_{k=1}^{\infty}2^{nk}\phi(2^{k-1})\frac{\mu(B(0,2^kt))}{m(B(0,2^kt))}\nonumber\\
&\leq&m(B(0,1))\left(\phi(0)+\sum_{k=1}^{\infty}2^{nk}\phi(2^{k-1})\right)\sup_{r>\delta}\frac{\mu(B(0,r))}{m(B(0,r))}\label{maximalr1}
\end{eqnarray}
As $\phi$ is radial and radially decreasing, we have
\begin{eqnarray*}
\int_{1}^{\infty}\phi(r)r^{n-1}dr
&=&\sum_{k=1}^{\infty}\int_{2^{k-1}}^{2^k}\phi(r)r^{n-1}dr\\
&\geq&\sum_{k=1}^{\infty}\phi(2^k)\int_{2^{k-1}}^{2^k}r^{n-1}\:dr\\
&=&\sum_{k=1}^{\infty}\phi(2^k)\frac{2^{nk}-2^{n(k-1)}}{n}\\
&=&\frac{2^{-n}-2^{-2n}}{n}\sum_{k=1}^{\infty}\phi(2^k)2^{n(k+1)}.
\end{eqnarray*}
Integrability of $\phi$ now implies that
\begin{equation*}
\sum_{k=1}^{\infty}2^{nk}\phi(2^{k-1})<\infty.
\end{equation*}
Hence, we conclude from inequality (\ref{maximalr1}) that for all $t>\delta$
\begin{equation}\label{maximalr2}
\mu\ast\phi_t(0)\leq C_{n,\phi}\sup_{t>\delta}\frac{\mu(B(0,t))}{m(B(0,t))}=C_{n,\phi}\sup_{t>\delta}M(t),
\end{equation}
where
\begin{equation*}
C_{n,\phi}=m(B(0,1))\left(\phi(0)+\sum_{k=1}^{\infty}2^{nk}\phi(2^{k-1})\right)<\infty.
\end{equation*}
As $\mu$ satisfies (\ref{growth}),  there exist positive constants $C$ and $r_0$ such that for all $r\geq r_0$,
\begin{equation*}
\frac{\mu(B(0,r))}{m(B(0,r))}\leq C.
\end{equation*}
On the other hand, for all $r\in(\delta,r_0)$,
\begin{equation*}
\frac{\mu(B(0,r))}{m(B(0,r))}\leq \frac{\mu(B(0,r))}{m(B(0,\delta))}\leq \frac{\mu(B(0,r_0))}{m(B(0,\delta))}.
\end{equation*}
Therefore,
\begin{equation*}
\sup_{r>\delta}M(r)=\sup_{r>\delta}\frac{\mu(B(0,r))}{m(B(0,r))}<\infty.
\end{equation*}
Existence of $v$ on $(t_0,\infty)$ and boundedness of $v$ on $(\delta,\infty)$ now follow from  boundedness of $M$ on $(\delta,\infty)$ and (\ref{maximalr2}).  This completes the proof.
\end{proof}
\begin{remark}
The condition (\ref{tauberian}) on $\phi$ is necessary in this case as well and can be seen from Example \ref{nec1} by noting that the positive measure $d\mu(x)=f(x)dx$ described there satisfies (\ref{growth}), in fact, 
$\mu(B(0,r))\leq 3\:m(B(0,r))$, for $r\in (0,\infty)$.
\end{remark}
As an application of Theorem \ref{mt1} we now suggest an alternative proof of a result \cite[Theorem 4]{W} (see also \cite[Theorem 4]{G} for the case $n=1$) regarding nonnegative solutions of the heat equation
\begin{equation*}
\Delta u(x,t)=u_t(x,t),\:\:\:\:(x,t)\in\R^{n+1}_+,
\end{equation*}
where $\Delta$ is the Laplacian of $\R^n$. It is known \cite[P.93-99]{W1} that if $u$ is a nonnegative solution of the heat equation in $\R^{n+1}_+$ then there exists a unique positive Radon measure $\mu$ on $\R^n$ such that
\begin{equation*}
u(x,t)=\mu*h_t(x)=\int_{\R^n}h_t(x-y)\:d\mu(y),\:\:\:\: (x,t)\in\R^{n+1}_+,
\end{equation*}
where $h_t$ is the heat kernel (see (\ref{heat})). The measure $\mu$ is called the boundary measure of $u$.
\begin{corollary}\label{heatpositive}
Let $u$ be a nonnegative solution of the heat equation with boundary measure $\mu$. If for some $x_0\in\R^n$ and $L\in [0,\infty)$
\begin{equation*}
\lim_{t\to 0} u(x_0,t)=L,
\end{equation*}
then
\begin{equation*}
D_{sym}\mu(x_0)=L.
\end{equation*}
\end{corollary}
\begin{proof}
We consider the Gaussian $w$ given in Example \ref{example}, ii).
Clearly, $w$ is a strictly positive, radial, radially decreasing function on $\R^n$. Moreover, $\|w\|_1=1$ \cite[P.9]{SW} and $w$ satisfies the comparison condition (\ref{comp}) ( see Example \ref{example}, ii)). To apply Theorem \ref{mt1} all we need is to show that $w$ satisfies (\ref{tauberian}). Now, for all $y\in\R$,
\begin{eqnarray*}
\int_{\R^n}w(x)\|x\|^{iy}dx
&=&(4\pi)^{-\frac{n}{2}}\int_{\R^n}e^{-\frac{\|x\|^2}{4}}\|x\|^{iy}\:dx\\
&=&(4\pi)^{-\frac{n}{2}}\omega_{n-1}\int_{0}^{\infty}e^{-\frac{r^2}{4}}r^{iy}r^{n-1}\:dr\\
&=&(4\pi)^{-\frac{n}{2}}\omega_{n-1}2^{n-1+iy}\int_{0}^{\infty}e^{-t}t^{\frac{iy+n}{2}-1}\:dt\\
&&\:\:\:\:\:\:\:\:(\text{by the substitution $r=2\sqrt{t}$}) \\
&=&(4\pi)^{-\frac{n}{2}}\omega_{n-1}2^{n-1+iy}\Gamma\left(\frac{n+iy}{2}\right),
\end{eqnarray*}
which is nonzero. As
\begin{equation*}
\lim_{t\to 0}u(x_0,t)=\lim_{t\to 0}\mu\ast h_t(x_0)=\lim_{t\to 0}\mu\ast w_{\sqrt{t}}(x_0,t)=\lim_{t\to 0}\mu\ast w_t(x_0)=L,
\end{equation*}
the proof follows by Theorem \ref{mt1}.
\end{proof}
Our next result is a generalization of the result of Repnikov-Eidelman (Theorem \ref{heatrep}) alluded to in the introduction.
\begin{theorem}\label{mtheat}
Suppose $\phi\in L^1(\R^n)$, $\psi\in L^1(\R^n)$ are radial functions with
\begin{equation}\label{int1}
\int_{\R^n}\phi(x)\:dx=\int_{\R^n}\psi(x)\:dx=1.
\end{equation}
Further assume that $\phi$ satisfies the condition (\ref{tauberian}). Suppose $f\in L^{\infty}(\R^n)$ is such that for some $x_0\in\R^n$ and  $L\in\C$,
\begin{equation*}
\lim_{t\to \infty} f\ast\phi_t(x_0)=L.
\end{equation*}
Then
\begin{equation*}
\lim_{t\to \infty} f\ast\psi_t(x_0)=L.
\end{equation*}
\end{theorem}
 \begin{proof}
Using polar coordinates we write
 \begin{eqnarray}\label{fphi}
 f\ast\phi_t(x_0)&=&t^{-n}\int_{\R^n}f(x_0-x)\phi\left(\frac{x}{t}\right)\:dx\nonumber\\
 &=&t^{-n}\int_{0}^{\infty}\int_{S^{n-1}}f(x_0-r\omega)\:d\sigma(\omega)\phi\left(\frac{r}{t}\right)r^{n-1}\:dr\nonumber\\
 &=&\int_{0}^{\infty}f_0(r)\phi\left(\frac{r}{t}\right)\left(\frac{r}{t}\right)^n\:\frac{dr}{r},
\end{eqnarray}
where \begin{equation*}
f_0(r)=\int_{S^{n-1}}f(x_0-r\omega)\:d\sigma(\omega),\:\:\:\:r> 0,
\end{equation*}
with $\sigma$ being the rotation invariant measure on the unit sphere $S^{n-1}$. Clearly, $f_0$ is a bounded function on $(0,\infty)$. We set
 \begin{equation*}
 g_{\phi}(s)=s^{-n}\phi(s^{-1}),\:\:\:\:s>0.
 \end{equation*}
 From (\ref{fphi}) we get the relation
 \begin{equation}\label{gphi}
 f\ast\phi_t(x_0)=f_0\ast_{(0,\infty)}g_{\phi}(t),\:\:\:\:t>0.
 \end{equation}
 A similar computation shows that
\begin{equation}\label{gpsi}
 f\ast\psi_t(x_0)=f_0\ast_{(0,\infty)}g_{\psi}(t),\:\:\:\:t>0,
\end{equation}
where
\begin{equation*}
g_{\psi}(s)=s^{-n}\psi(s^{-1}),\:\:\:\:s>0.
\end{equation*}
Since $\phi$, $\psi$ are radial and integrable functions on $\R^n$ it follows that $g_{\phi}$ and $g_{\psi}$ belong to the space $L^1((0,\infty), \frac{ds}{s})$. Moreover, by (\ref{int1})
\begin{equation}\label{ft0}
\int_0^{\infty}g_{\phi}(s)\:\frac{ds}{s}=\int_0^{\infty}g_{\psi}(s)\:\frac{ds}{s}=\frac{1}{\omega_{n-1}}.
\end{equation}
A simple calculation as in the proof of Theorem \ref{mt1}, shows that the Fourier transform of $g_{\phi}$ on the multiplicative group  $(0,\infty)$ satisfies
\begin{equation*}
\what g_{\phi}(y)=\int_{0}^{\infty}g_{\phi}(s)s^{-iy}\:\frac{ds}{s}=\frac{1}{\omega_{n-1}}\int_{\R^n}\phi(x)\|x\|^{iy}\:dx\neq 0,
\end{equation*}
for all $y\in\R$, as $\phi$ satisfies (\ref{tauberian}).
Using the equations (\ref{gphi}) and (\ref{ft0}), it follows from the hypothesis that
\begin{equation*}
\lim_{t\to \infty}f_0\ast_{(0,\infty)}g_{\phi}(t)=\lim_{t\to \infty} f\ast\phi_t(x_0)=L=L\omega_{n-1}\what{g_{\phi}}(0).
\end{equation*}
From Wiener's Tauberian theorem (Theorem \ref{wtt}) and (\ref{ft0}) it follows that
\begin{equation*}
\lim_{t\to \infty}f_0\ast_{(0,\infty)}g_{\psi}(t)=L\omega_{n-1}\what{g_{\psi}}(0)=L.
\end{equation*}
An application of the relation (\ref{gpsi}) completes the proof.
\end{proof}
\begin{remark}\label{repnikov}	
\begin{enumerate}
\item To deduce Theorem \ref{heatrep} from Theorem \ref{mtheat},
we choose $\phi=w$  and $\psi=m(B(0,1))^{-1}\chi_{B(0,1)}$. We have already shown  in the proof of Corollary \ref{heatpositive} that $\phi$ satisfies all the conditions of Theorem \ref{mtheat}. We observe that $\psi$ is a nonnegative, radial and integrable function on $\R^n$ with $\|\psi\|_1=1$. Hence, to deduce Theorem \ref{heatrep} it suffices for us to show that $\psi$ satisfies (\ref{tauberian}). Now, for all $y\in\R$
\begin{eqnarray*}
\int_{\R^n}\psi(x)\|x\|^{iy}\:dx&=& m(B(0,1))^{-1}\int_{B(0,1)}\|x\|^{iy}\:dx\\
&=& m(B(0,10)^{-1}\omega_{n-1}\int_{0}^{1}r^{iy+n-1}\:dr\\
&=& m(B(0,1))^{-1}\omega_{n-1}\frac{\Gamma(n+iy)\Gamma(1)}{\Gamma(n+1+iy)},
\end{eqnarray*}
which is nonzero.
Now, suppose $f\in L^{\infty}(\R^n)$ and $x_0\in\R^n$, $L\in\C$. Applying Theorem \ref{mtheat} twice, it follows that
\begin{equation*}
\lim_{t\to\infty} f\ast h_t (x_0)=\lim_{t\to\infty} f\ast w_{\sqrt{t}} (x_0)=\lim_{t\to\infty} f\ast w_t (x_0)=L,
\end{equation*}
if and only if
\begin{equation*}
\lim_{t\to\infty}f\ast \psi_t(x_0)=\lim_{t\to\infty}\frac{1}{m(B(x_0,t))}\int_{B(x_0,t)}f(x)\:dx=L.
\end{equation*}
This proves Theorem \ref{heatrep}.
\item We show by an example that condition (\ref{tauberian}) is necessary for the validity of Theorem \ref{mtheat} as well. Suppose $\phi\in L^1(\R^n)$ is a radial function such that
\begin{equation}\label{againint1}
\int_{\R^n}\phi(x)dx=1.
\end{equation}
Assume that there exists $y_0\in\R$ such that
\begin{equation}\label{ftzero}
\int_{\R^n}\phi(x)\|x\|^{iy_0}dx=0,
\end{equation}
From (\ref{againint1}), it is clear that $y_0$ is nonzero. Consider the function
\begin{eqnarray*}
f(x)&=&\|x\|^{iy_0},\:\:\:\:x\in\R^n\setminus\{0\},\\
&=&1,\:\:\:x=0.
\end{eqnarray*}
Then $f\in L^{\infty}(\R^n)$ and by (\ref{ftzero}), we have that for all $t\in (0,\infty)$
\begin{equation*}
f\ast\phi_t(0)=t^{-n}\int_{\R^n}\|x\|^{iy_0}\phi\left(\frac{x}{t}\right)\:dx
=t^{-n}\int_{\R^n}\|t\xi\|^{iy_0}\phi(\xi)t^n\:d\xi
= t^{iy_0}\int_{\R^n}\phi(\xi)\|\xi\|^{iy_0}\:d\xi
= 0,
\end{equation*}
and hence
\begin{equation*}
\lim_{t\to \infty} f\ast\phi_t(0)=0.
\end{equation*}
As in $(1)$, we again consider the function $\psi=m(B(0,1))^{-1}\chi_{B(0,1)}$. Then $\psi$ is nonnegative, radial with $\|\psi\|_1=1$.
We observe that
\begin{eqnarray*}
f\ast\psi_t(0)&=&m(B(0,1))^{-1}t^{-n}\int_{B(0,t)}f(x)\:dx\\
&=& m(B(0,1))^{-1}t^{-n}\int_{B(0,t)}\|x\|^{iy_0}\:dx\\
&=&\omega_{n-1}m(B(0,1))^{-1}t^{-n}\int_{0}^{t}r^{iy_0}r^{n-1}\:dr\\
&=& \omega_{n-1}m(B(0,1))^{-1}\frac{t^{iy_0}}{n+iy_0}.
\end{eqnarray*}
It follows that $f\ast\psi_t(0)$ does not converge to any limit as $t$ goes to infinity.
\end{enumerate}
\end{remark}
Since a bounded harmonic function $u$ on $\R^{n+1}_+$ is the Poisson integral of a unique boundary function $f\in L^{\infty}(\R^n)$ (see \cite[Theorem 2.5]{SW}), the following result is a simple consequence of Theorem \ref{mtheat}.
\begin{corollary}\label{bddhar}
Suppose $u$ is a bounded harmonic function on $\R^{n+1}_+$ with boundary function $f$. Then for $x_0\in \R^n$ and $L\in\C$,
\begin{equation*}
\lim_{y\to\infty}u(x_0,y)=L,
\end{equation*}
if and only if
\begin{equation*}
\lim_{r\to\infty}\frac{1}{m(B(x_0,r))}\int_{B(x_0,r)}f(x)dx=L.
\end{equation*}
\end{corollary}
\begin{proof}
We have
\begin{equation*}
u(x,y)=f\ast P_y(x),\:\:x\in\R^n,\: y\in (0,\infty).
\end{equation*}
The function $P_1$ is radial and positive with $\|P_1\|_1=1$ (see \cite[P.9]{SW}). We have for any $s\in\R$ (see \cite[Equation (24)]{Ru}),
 \begin{equation*}
 \int_{\R^n}(1+\|x\|^2)^{-\frac{n+1}{2}}\|x\|^{is}\:dx=c_n\frac{\Gamma\left(\frac{n+is}{2}\right)\Gamma\left(\frac{1-is}{2}\right)}{2\Gamma\left(\frac{n+1}{2} \right)}\neq 0.
 \end{equation*}
This shows that $P_1$ satisfies (\ref{tauberian}). We have already shown in Remark \ref{repnikov}, $(1)$ that the function $m(B(0,1))^{-1}\chi_{B(0,1)}$ also satisfies (\ref{tauberian}). Applying Theorem \ref{mtheat} twice, first with $\phi$=$P_1$ and then with $\phi=m(B(0,1))^{-1}\chi_{B(0,1)}$, we get the result.
\end{proof}

\section{Real hyperbolic spaces}
In this section we will apply Theorem \ref{mt1} and Theorem \ref{mtheat} in the context of real hyperbolic spaces and prove some analogous results for certain eigenfunctions of the Laplace-Beltrami operator. We start with a brief review of some basic facts about real hyperbolic spaces (see \cite{Da, sto}). We  consider the Poincare upper half space model of the $n$-dimensional real hyperbolic space
\begin{equation*}
\mathbb{H}^n=\{(x,y)\mid x\in \R^{n-1}, y\in (0,\infty)\},\:\:\:\:n\geq 2,
\end{equation*}
equipped with the standard hyperbolic metric $ds^2=y^{-2}(dx^2+dy^2)$. The boundary of $\mathbb{H}^n$ is identified with $\R^{n-1}$. The Laplace-Beltrami operator for $\mathbb H^n$ is given by the formula \cite[P. 176]{Da}
\begin{equation*}
\Delta_{\mathbb{H}^n}=y^2\left(\Delta_x+\frac{\partial^2}{\partial y^2}\right)-(n-2)y\frac{\partial}{\partial y}.
\end{equation*}
The expression for the corresponding Poisson kernel $\mathcal{P}$ is given by \cite[P. 76]{sto}
\begin{equation*}
\mathcal{P}\left(x,y\right)=c_n\frac{y^{n-1}}{(y^2+\|x\|^2)^{n-1}},\:\:\:\:(x,y)\in\mathbb{H}^n,
\end{equation*}
where $c_n$ is such that
\begin{equation*}
\int_{\R^{n-1}}\mathcal P (x,y)\:dx=1,\:\:\:\: y\in (0,\infty).
\end{equation*}
We note that if
\begin{equation*}
\phi (x)=\frac{c_n}{(1+\|x\|^2)^{n-1}}, \:\:\:\:x\in \R^{n-1},
\end{equation*}
then for $y\in (0,\infty )$
\begin{equation*}
\phi_y(x)=y^{-(n-1)}\phi\left(\frac{x}{y}\right)=\mathcal P(x,y).
\end{equation*}
It is known that various classes of harmonic functions on $\mathbb H^n$ are Poisson integral of functions or measures defined on the boundary $\R^{n-1}$. One such class is the collection of nonnegative harmonic functions. It is known that given any nonnegative harmonic function $u$ on $\mathbb H^n$ there exists a unique positive measure $\mu$ on $\R^{n-1}$ and a nonnegative constant $A$ such that
\begin{equation*}
u(x,y)=Ay^{n-1}+\int_{\R^{n-1}}\mathcal P (x-\xi, y )\:d\mu (\xi).
\end{equation*}
(see \cite[P. 113]{sto}). It turns out that more general eigenfunctions of $\Delta_{\mathbb H^n}$ can be obtained by considering the generalized Poisson kernel.
For $\lambda\in\C$, the generalized Poisson kernel corresponding to $\lambda$ is given by the formula
\begin{equation}\label{genpoisson}
\mathcal{P}_{\lambda}(x,y)=\left(\frac{\mathcal{P}(x,y)}{\mathcal{P}(0,1)}\right)^{\frac{1}{2}-\frac{i\lambda}{n-1}}=\left[\frac{y^{n-1}}{(y^2+\|x\|^2)^{n-1}} \right]^{\frac{1}{2}-\frac{i\lambda}{n-1}},\:\:\:\:(x,y)\in\mathbb{H}^n.
\end{equation}
It is well-known that for $\lambda\in\C$, the function $\mathcal P_{\lambda}$ is an eigenfunctions of $\Delta_{\mathbb H^n}$ and satisfies the following (see \cite[P. 654]{ADY})
\begin{equation*}
\Delta \mathcal{P}_{\lambda}=-(\lambda^2+\rho^2)\mathcal{P}_{\lambda},\:\:\:\: \rho=\frac{n-1}{2}.
\end{equation*}
From the explicit expression (\ref{genpoisson}) it is easy to see that that for $\text{Im}(\lambda)\in (0,\infty)$, $\mathcal P_{\lambda} (\cdot,y)\in L^1 (\R^{n-1})$,  for all $y\in (0,\infty )$ and
\begin{equation}\label{p1}
\int_{\R^{n-1}}\mathcal{P}_{\lambda}(x,1)\:dx= \frac{\bf c(-\lambda)}{c_n},
\end{equation}
where $\bf c (\lambda)$ is Harish-Chandra $\bf c$-function for $\mathbb{H}^n$ and is given by
\begin{equation}\label{cfun}
{\bf c}(\lambda)=2^{n-1-2i\lambda}\frac{\Gamma(2i\lambda)\Gamma(n/2)}{\Gamma(\frac{n-1}{2})\Gamma(1/2+i\lambda)},\:\:\:\: \text{Im}(\lambda)<0.
\end{equation}
It is clear from the formula that the $\bf c$-function has no pole or zero in the left half plane and hence we can normalize $\mathcal P_{\lambda}$ to define
\begin{equation}\label{normalised}
P_{\lambda}(x,y)=d_{\lambda}\mathcal{P}_{\lambda}(x,y),\:\:\:\:d_{\lambda}=\frac{c_n}{{\bf c}(-\lambda)},\:\:\: \text{Im}(\lambda)>0.
\end{equation}
Using the expression (\ref{genpoisson}) we have the following important observation
\begin{eqnarray}\label{dla2}
P_{\lambda}(x,y)&=& d_{\lambda}\left[\frac{y^{n-1}}{(y^2+\|x\|^2)^{n-1}} \right]^{\frac{1}{2}-\frac{i\lambda}{n-1}}\nonumber\\
&=& y^{\frac{n-1}{2}+i\lambda}y^{-(n-1)}\frac{d_{\lambda}}{(1+\|y^{-1}x\|^2)^{\frac{n-1}{2}-i\lambda}}\nonumber\\
&=& y^{\rho+i\lambda}y^{-(n-1)}P_{\lambda}\left(\frac{x}{y},1\right)\nonumber\\
&=& y^{\rho+i\lambda}\left(\psi^{\lambda}\right)_y(x),
\end{eqnarray}
where
\begin{equation}\label{psilam}
\psi^{\lambda}(x)=P_{\lambda}(x,1),\:\:\:\:\left(\psi^{\lambda}\right)_y(x)=y^{-(n-1)}\psi^{\lambda}\left(\frac{x}{y}\right).
\end{equation}
It follows from (\ref{p1}) and (\ref{normalised}) that for $\text{Im}(\lambda)\in (0,\infty )$,
\begin{equation}\label{intpsilam}
\int_{\R^{n-1}}\psi^{\lambda}(x)\:dx=1.
\end{equation}
Hence, for all $y>0$,
\begin{equation*}
\int_{\R^{n-1}}P_{\lambda}(x,y)\:dx=y^{\rho+i\lambda},\:\:\:\: \text{Im}(\lambda)>0.
\end{equation*}
Using the fact that ${\bf c} (-i\rho )=1$, it follows that for all $(x,y)\in \mathbb H^n$
\begin{equation*}
P_{i\rho}(x,y)=\mathcal P (x,y).
\end{equation*}
For $\text{Im}(\lambda)>0$, we define the Poisson transform of a complex measure (or a signed measure) $\mu$ on $\R^{n-1}$ as the convolution
\begin{equation}\label{ptrans}
P_{\lambda}\mu(x,y)=\int_{\R^{n-1}}P_{\lambda}(x-\xi,y)\:d\mu(\xi),\:\:\:\:(x,y)\in\mathbb H^n,
\end{equation}
whenever the integral exists. It follows that $P_{\lambda}\mu$ also satisfies the eigenvalue equation
\begin{equation*}
\Delta_{\mathbb H^n}P_{\lambda}\mu=-(\lambda^2+\rho^2)P_{\lambda}\mu.
\end{equation*}
The relation (\ref{dla2}) implies that $P_{\lambda}\mu$ can be rewritten as
\begin{equation}\label{psilamcon}
P_{\lambda}\mu(x,y)=y^{\rho+i\lambda}\left(\mu\ast\left(\psi^{\lambda}\right)_y\right)(x),\:\:\:\:x\in\R^{n-1},\: y>0.
\end{equation}
From (\ref{genpoisson}), (\ref{cfun}) and
(\ref{psilam}) it is easy to see that $\psi^{\lambda}(x)$ is  positive for all $x\in\R^{n-1}$, if and only if $\lambda$ is equal to $i\beta$, for some $\beta\in (0,\infty)$. In this case, $P_{\lambda}\mu$ is a nonnegative eigenfunction with eigenvalue $(\beta^2-\rho^2)$, whenever $\mu$ is a positive measure. In fact, we have the following characterization of nonnegative eigenfunctions of $\Delta_{\mathbb H^n}$ \cite[Theorem 7.11]{DR}.
\begin{lemma}\label{dr}
If $u$ is a nonnegative eigenfunction of  $\Delta_{\mathbb{H}^n}$ with eigenvalue $\beta^2-\rho^2$ for some $\beta\in (0,\infty)$, then there exists a unique positive Radon measure $\mu$ on $\R^{n-1}$ and a constant $C\geq 0$, such that
\begin{equation}\label{drexp}
u(x,y)=Cy^{\beta+\rho}+P_{i\beta}\mu(x,y),
\end{equation}	
for all $(x,y)\in\mathbb{H}^n$.
\end{lemma}
The measure $\mu$ in the theorem above will be called the boundary measure of the eigenfunction $u$.
We are now ready to prove an analogue of Rudin's result (Theorem \ref{rud}) for nonnegative eigenfunctions of $\Delta_{\mathbb H^n}$.
\begin{theorem}\label{mtrealhyper}
Suppose $u$ is a nonnegative eigenfunction of $\Delta_{\mathbb{H}^n}$ with boundary measure $\mu$ and eigenvalue $\beta^2-\rho^2$ for some $\beta>0$. If there exists $x_0\in\R^{n-1}$ and $L\in[0,\infty)$ such that
\begin{equation}\label{eigenlim}
\lim_{y\to 0}y^{\beta-\rho}u(x_0,y)=L,
\end{equation}
then $D_{sym}\mu(x_0)=L$.
\end{theorem}
\begin{proof}
As $\mu$ is the boundary measure of $u$, the expressions (\ref{drexp}) and (\ref{psilamcon}) imply that
\begin{equation*}
u(x,y)=Cy^{\beta+\rho}+y^{\rho-\beta}\left(\mu\ast\left(\psi^{i\beta}\right)_y\right)(x),\:\:\:\:x\in\R^{n-1},\:y>0.
\end{equation*}
Hence,
\begin{equation*}
y^{\beta-\rho}u(x,y)=Cy^{2\beta}+\left(\mu\ast\left(\psi^{i\beta}\right)_y\right)(x)\:\:\:\:x\in\R^{n-1},\:y>0.
\end{equation*}
By the hypothesis (\ref{eigenlim}) and the fact that $\beta>0$, it follows that
\begin{equation*}
\lim_{y\to 0}\left(\mu\ast\left(\psi^{i\beta}\right)_y\right)(x_0)=\lim_{y\to 0}y^{\beta-\rho}u(x_0,y)=L.
\end{equation*}
We recall from (\ref{psilam}) that
\begin{equation*}
\psi^{i\beta}(x)=d_{i\beta}\left(\frac{1}{1+\|x\|^2}\right)^{\rho+\beta},\:x\in\R^{n-1}.
\end{equation*}
It is clear from the expression above that $\psi^{i\beta}$ is a strictly positive, radial and radially decreasing function on $\R^{n-1}$. Moreover, by Example \ref{example}, we have that $\psi^{i\beta}$ satisfies the comparison condition (\ref{comp}). We now need to check that $\psi^{i\beta}$ satisfies the condition (\ref{tauberian}) of Theorem \ref{mt1}. To do this, we will need the following well-known formula
\begin{equation}\label{betafn}
\int_0^{\pi/2}(\sin\theta )^z(\cos\theta)^w\:d\theta= \frac{\Gamma\left(\frac{z+1}{2}\right)\Gamma\left(\frac{w+1}{2}\right)}{2\Gamma\left(\frac{z+w+2}{2}  \right)},\:\:\:\:\text{Re}(z)>-1,\:\text{Re}(w)>-1.
\end{equation}
Now, for any $\lambda\in\C$ with $\text{Im}(\lambda)>0$, and $s\in\R$ we have
\begin{eqnarray}\label{psilmtb}
\int_{\R^{n-1}}\psi^{\lambda}(x)\|x\|^{is}\:dx
&=&d_{\lambda}\int_{\R^{n-1}}\left(\frac{1}{1+\|x\|^2}\right)^{\rho-i\lambda}\|x\|^{is}\:dx\nonumber\\
&=&d_{\lambda}\omega_{n-2}\int_{0}^{\infty}\left(\frac{1}{1+r^2}\right)^{\rho-i\lambda}r^{is}r^{n-2}\:dr\nonumber\\
&=&d_{\lambda}\omega_{n-2}\int_{0}^{\pi/2}\left(\frac{1}{\sec^2\theta}\right)^{\rho-i\lambda}(\tan\theta)^{n-2+is}(\sec\theta )^2 \:d\theta\nonumber\\
&&\:\:\:\:\:\:\text{(using the substitution $r=\tan\theta$)}\nonumber\\
&=&d_{\lambda}\omega_{n-2}\int_{0}^{\pi/2}
(\cos\theta)^{(2\rho-2i\lambda-n-is)}(\sin\theta)^{n-2+is}\:d\theta\nonumber\\
&=&d_{\lambda}\omega_{n-2}\frac{\Gamma\left(\frac{2\rho-2i\lambda-n-is+1}{2}\right)\Gamma\left(\frac{n-1+is}{2}\right)}{2\Gamma(\rho-i\lambda)},
\end{eqnarray}
where the last equality follows from (\ref{betafn}) as $\text{Im}(\lambda)\in (0,\infty)$ . As, the expression on the right-hand side of (\ref{psilmtb}) is nonzero, it follows that $\psi^{i\beta}$ satisfies (\ref{tauberian}). In view of (\ref{intpsilam}), the proof now follows simply by applying Theorem \ref{mt1}.
\end{proof}

The last topic we are going to discuss is related to Theorem \ref{heatrep}. It is known that the exact analogue of the result of Repnikov-Eidelman (Theorem \ref{heatrep}) is false on $\mathbb H^n$ (see \cite{NRS, Re2}). However, an analogue of Corollary \ref{bddhar} (which we view as a variant of Theorem \ref{heatrep}) can be proved for $\mathbb H^n$.

We define for $\lambda\in\C$ and $f$ a measurable function on $\R^{n-1}$, the Poisson transform $P_{\lambda}f$ by a convolution analogous to (\ref{ptrans})
\begin{equation*}
P_{\lambda}f(x,y)=\int_{\R^{n-1}}P_{\lambda}(x-\xi,y)f(\xi)\:d\xi,\:\:(x,y)\in\mathbb{H}^n,
\end{equation*}
whenever the integral makes sense. We note that for $\text{Im}(\lambda)\in (0,\infty)$, the kernel $P_{\lambda}(\cdot ,y)$ is integrable for every $y\in (0,\infty)$ and hence the Poisson transform $P_{\lambda}f$ is well defined for $f\in L^{\infty}(\R^{n-1})$.
We will now prove an analogue of Corollary \ref{bddhar} for certain eigenfunctions of $\Delta_{\mathbb{H}^n}$. In order to do this,
we will need the following characterization of eigenfunctions of $\Delta_{\mathbb H^n}$ (see \cite[Theorem 3.6]{B}).
\begin{lemma}
	Suppose $u$ is an eigenfunction of $\Delta_{\mathbb{H}^n}$ with eigenvalue $-(\lambda^2+\rho^2)$, where $\text{Im}(\lambda)\in (0,\infty)$. Then $u=P_{\lambda}f$ for some $f\in L^{\infty}(\R^{n-1})$ if and only if
	\begin{equation}\label{hardynorm}
	\sup_{y>0}y^{\text{Im}(\lambda)-\rho}\|u(.,y)\|_{L^{\infty}(\R^{n-1})}<\infty.
	\end{equation}
\end{lemma}
We shall call $f$ to be the boundary function of $u$.
The following result, for $\lambda=i\rho$ can be thought of as an exact analogue of Corollary \ref{bddhar}.
\begin{theorem}\label{largetime}
	Suppose $u$ is an eigenfunction of $\Delta_{\mathbb{H}^n}$ with eigenvalue $-(\lambda^2+\rho^2)$, where $\text{Im}(\lambda)>0$. Further suppose that $u$ satisfies ($\ref{hardynorm}$) and $f$ is the boundary function of $u$. Then for $x_0\in\R^{n-1}$ and $L\in\C$
	\begin{equation*}
	\lim_{y\to\infty}y^{-(\rho+i\lambda)}u(x_0,y)=L
	\end{equation*}
	if and only if
	\begin{equation*}
	\lim_{r\to\infty}\frac{1}{m(B(x_0,r))}\int_{B(x_0,r)}f(x)dx=L.
	\end{equation*}
\end{theorem}
\begin{proof}
	Since $f$ is the boundary function of $u$, we have
	\begin{equation*}
	u(x,y)=P_{\lambda}f(x)=\int_{\R^{n-1}}P_{\lambda}(x-\xi,y )f(\xi)\:d\xi,\:\:(x,y)\in\mathbb{H}^n.
	\end{equation*}
The equation (\ref{dla2}) now implies that
	\begin{equation*}
y^{-(\rho+i\lambda)}u(x_0,y)=\left(f\ast\left(\psi^{\lambda}\right)_y\right)(x_0).
	\end{equation*}
We also have
	\begin{equation*}
	\frac{1}{m(B(x_0,r))}\int_{B(x_0,r)}f(x)dx=f\ast\left(m(B(0,1))^{-1}\chi_{B(0,1)}\right)_r(x_0).
	\end{equation*}
	Now, (\ref{psilmtb}) shows that $\psi^{\lambda}$ satisfies (\ref{tauberian}). Also, from (\ref{psilam}) and (\ref{intpsilam}), we have that $\psi^{\lambda}$ is radial and is of integral one. We have already observed in Remark \ref{repnikov} that $m(B(0,1))^{-1}\chi_{B(0,1)}$ also obeys (\ref{tauberian}). Application of Theorem \ref{mtheat} twice, first with $\phi$=$\psi^{\lambda}$ and then with $\phi=m(B(0,1))^{-1}\chi_{B(0,1)}$ finishes the proof.
\end{proof}
\begin{remark}
It is known that the analogue of Theorem \ref{mtrealhyper} (for $\beta=\rho$) is false for complex hyperbolic spaces (see \cite[P.78]{Ru2}). However, it is not clear to us, at the moment, whether the exact analogue of Theorem \ref{largetime} holds for complex hyperbolic spaces.
\end{remark}
\section*{acknowledgements}
The author would like to thank Swagato K. Ray for suggesting this problem and for many
useful discussions during the course of this work.

\end{document}